\documentclass[12pt]{amsart}

\usepackage{color}

\newcommand{\edit}[1]{{\color{red}{$\clubsuit$#1$\clubsuit$}}}

\newcommand{\sm}{\setminus}

\usepackage{graphicx}

\usepackage[active]{srcltx}

\newcommand{\Frechet}{Fre\`chet}

\usepackage[active]{srcltx}

\newcommand{\Poincare}{Poincar\'e}

\usepackage{latexsym}
\usepackage{graphicx}

\renewcommand{\Re}{{\operatorname{Re}\,}}
\renewcommand{\Im}{{\operatorname{Im}\,}}

\renewcommand{\epsilon}{\varepsilon}

\newcommand{\var}{{\operatorname{Var}}}

\newcommand{\szego}{Szeg\H{o} }

\newcommand{\kahler}{K\"ahler }

\newcommand{\wt}{\widetilde}
\newcommand{\wh}{\widehat}
\newcommand{\PP}{{\mathbb P}}

\newcommand{\R}{{\mathbb R}}
\newcommand{\C}{{\mathbb C}}

\newcommand{\CP}{\C\PP}
\renewcommand{\d}{\partial}
\newcommand{\dbar}{\bar\partial}
\newcommand{\ddbar}{\partial\dbar}
\newcommand{\U}{{\rm U}}

\newcommand{\E}{{\mathbf E}}

\newcommand{\half}{{\textstyle \frac 12}}
\newcommand{\vol}{{\operatorname{Vol}}}

\renewcommand{\phi}{\varphi}

\newcommand{\go}{\mathfrak}

\newcommand{\bcal}{\mathcal{B}}
\newcommand{\ccal}{\mathcal{C}}

\newcommand{\ecal}{\mathcal{E}}

\newcommand{\lcal}{\mathcal{L}}

\newcommand{\ncal}{\mathcal{N}}
\newcommand{\ocal}{\mathcal{O}}
\newcommand{\pcal}{\mathcal{P}}

\newcommand{\scal}{\mathcal{S}}

\newcommand{\al}{\alpha}

\newcommand{\ga}{\gamma}

\newcommand{\ep}{\varepsilon}
\newcommand{\de}{\delta}

\newtheorem{theo}{{\sc Theorem}}[section]

\newtheorem{defin}{{\sc Definition}}

\newtheorem{cor}[theo]{{\sc Corollary}}

\newtheorem{lem}[theo]{{\sc Lemma}}
\newtheorem{prop}[theo]{{\sc Proposition}}

\newenvironment{rem}{\medskip\noindent{\it Remark:\/} }{\medskip}

\title[Quantum ergodic sequences and equilibrium measures
 ] {Quantum ergodic sequences and equilibrium measures}

\author{Steve Zelditch }
\address{Department of Mathematics, Northwestern  University, Evanston, IL 60208, USA}

\email{zelditch@math.northwestern.edu}

\thanks{Research partially supported by NSF grant and DMS-1541126
and by the Stefan Bergman trust  .}

\begin{document}

\begin{abstract} We generalize the  definition of a ``quantum ergodic sequence'' of sections of ample  line bundles $L \to M$ from the case of positively curved
Hermitian metrics $h$ on $L$ to general smooth metrics. A  choice of smooth Hermitian metric $h$ on $L$ and  a Bernstein-Markov measure $\nu$ on $M$ induces an inner product 
on $H^0(M, L^N)$. When $||s_N||_{L^2} =1$, quantum ergodicity is
the condition that  $|s_N(z)|^2 d\nu \to
d\mu_{\phi_{eq}} $ weakly, where $d\mu_{\phi_{eq}} $ is the equilibrium measure associated to $(h, \nu)$. 
The main results are that normalized logarithms $\frac{1}{N} \log |s_N|^2$
of quantum ergodic sections tend to the equilibrium potential, and that
random orthonormal bases of $H^0(M, L^N)$ are quantum ergodic.

\end{abstract}

\maketitle

One of the principal themes of `stochastic \kahler geometry' is the asymptotic
equilibrium distribution of zeros of random holomorphic fields on \kahler manifolds $(M, J, \omega).$ A basic example is when $M = \CP^1 = \C \cup \{\infty\}$,  the
Riemann sphere, and when the random fields are holomorphic polynomials
$p_N$ of degree $N$. The zero set $Z_{p_N} = \{z: p_N(z) = 0\} $ is a random
set of $N $ points $\vec \zeta = \{\zeta_1, \dots, \zeta_N\}$ on $\CP^1$ and is encoded by the  empirical measure
\begin{equation} \label{mu} \mu_{\vec \zeta}: = \frac{1}{N} [Z_{p_N}] : =  \frac{1}{N} \sum_{j=1}^{N} \delta_{\zeta_j}. \end{equation}
Here, $[Z_{p_N}]$ is the geometer's notation for the normalized current of integration 
over $Z_{p_N}$. 
Given a weight $e^{- \phi}$ and a suitable measure $d\nu$ on $\CP^1$,
one defines the inner product
${\rm Hilb}_N(\phi, \nu)$ on the space $\pcal_N$ of polynomials of degree $N$ by
$$||p_N||^2_{{\rm Hilb}_N(\phi, \nu)} = \int_{\C} |p_N(z)|^2 e^{-N \phi} d\nu(z), $$
and this inner product induces a Gaussian measure on $\pcal_N$. The asymptotic equilibrium distribution of zeros is the statement that
for a random sequence $\{p_N\}$ of polynomials of increasing degree,
the empirical measures $[Z_N]$ almost surely tend to the weighted equilibrium measure $d \mu_{eq}$ corresponding to $(\phi, \nu)$. 

The asymptotic  equilibrium distribution of zeros is by now a very general
result that holds for Gaussian random holomorphic sections of line bundles
over \kahler manifolds with respect to weights and measures satisfying some quite weak conditions. Polynomials of degree $N$ on  $\C$ generalize
to the space $H^0(M, L^N)$ of holomorphic sections of the $N$th power of 
an ample line bundle $L \to M$ over any Riemann surface, or over any \kahler manifold of any (complex) dimension $m$. The weight $e^{-\phi}$ is regarded as a Hermitian metric $h$ on $L$. The geometric language is useful  not only for putting the polynomial problem in a general context but also for indicating the proper assumptions on the weights and measures, as well as the definition of the associated equilibrium measure. The almost sure equilibrium distribution of zeros of random sequences $\{s_N\}$ of holomorphic sections of degree $N$ follows from the fact that
the associated potentials $\frac{1}{N} \log |s_N(z)|_{h^N}^2$ tend almost surely to the equilibrium potential $\phi_{eq}$.

The first  purpose of this survey is to review the results on asymptotic
equilibrium distribution of zeros. In a sense, it is a universal result that has been developed alongside  generalizations of the notion of equilibrium potential and measure. It is closely related to asymptotics of Bergman kernels relative to quite general weights and measures. A second purpose of this survey is to explain the notion of a
`quantum ergodic sequence' of sections and moreover to generalize that notion  to 
the same weights and measures for which one has  \begin{equation} \label{uN} u_N:= \frac{1}{N} \log |s_N(z)|_{h^N}^2
\to \phi_{eq}\;\; \rm{in}\; L^1(M, dV)\end{equation} for any
volume measure $dV$.  The main new result of this survey is a proof that random 
sequences of sections are quantum ergodic and that normalized logarithms of quantum ergodic sequences tend to $\phi_{eq}$. The proofs mainly consist in combining the proof in \cite{ShZ99} of this result in the setting of positive Hermitian line bundles together with the results of Berman, Boucksom,  and  Witt-Nystrom \cite{BB10, BBWN11} on asymptotics of Bergman kernels with respect to rather general weights and measures. A quite general proof of  \eqref{uN} for normalized logarithms  of random sequences of polynomials of increasing degree
 has recently been proved by Bloom-Levenberg \cite[Theorems 4.1-4.2]{BL15}
in the `local setting' of polynomials on $\C^{m}$. The  `ergodic' proof we give is rather different and the generalized notion of quantum ergodic seems to us of independent interest.

\subsection{Historical background}

Before stating definitions and results, let us  try to put the problems into context. The asymptotic equilibrium distribution of zeros  was first proved
in the special case  of a positive Hermitian  line bundle $(L, h)$ over a \kahler manifold $(M, \omega)$, where 
 the curvature form of $h$ is the \kahler form $\omega$ \cite{ShZ99} (see also \cite{NV98} for genus one surfaces in
dimension one).  In this case, 
the equilibrium measure is simply the volume form $dV_{\omega} = \frac{\omega^m}{m!}$ of the metric. The next result, at least as known to the author, occurred 
in the almost opposite case of Kac-type polynomial ensembles,  where $h \equiv 1, $  $M = \C$ and $\nu$ is an analytic measure 
on an analytic plane domain or its boundary \cite{ShZ03}. It was in this setting that contact was made with the classical notion of equilibrium measure, and the pair of results  
suggested to the authors that equilibrium distribution of zeros should be a universal kind of result. T.  Bloom 
\cite{Bl05} shed new light on the results by pointing  out the role of the complex Green's function and extremal subharmonic functions in the equlibrium result. His article also introduced Bernstein-Markov measures $d\nu$ as the most general framework for defining the Gaussian random sections.   In \cite{Ber09}, R. Berman defined equilibrium potentials and measures for general metrics on an ample line bundle  over a \kahler manifold of any dimension,
and proved that \eqref{mu} tends to the equilibrium measure in this generality.  In \cite{ZeiZ10}, the equilibrium distribution of zeros for random polynomials with respect to general smooth weights and Bernstein-Markov measures  was derived from a large deviations
principle for the empirical measures \eqref{mu} in complex dimension one; the author later extended the result to any Riemann surface. 
There are many other  articles proving results like \eqref{uN} in different settings, including
\cite{DS06, BlS07,CM15,DMM16,Bay16}. The equilibrium distribution of zeros 
 in dimension one is also reminscent of equilibrium distribution of eigenvalues or for the points of a  Coulomb gas or other determinantal  point processes (see e.g. \cite{AHM11}), but zero point processes in complex dimension one  are more complicated than Coulomb gases and are almost never determinantal. The proofs that zeros are equidistributed by the equilibirum measure are  quite different than the proofs for Coulomb gases. Moreover, the higher dimensional generalizations required advances in the theory of Bergman kernels and equilibrium measures \cite{GZ05,Ber09,Bl09}.

\subsection{Quantum ergodic sequences}

We now explain what is meant by a  {\it quantum ergodic}  sequence $\{s_N\}$ of holomorphic sections of powers $L^N$ of an  ample line bundle $L \to M$  over a
\kahler manifold $(M, \omega)$ of dimension $m$. It is a deterministic notion.

  In the case $M = \CP^m$, $\{s_N\}$ is a sequence of homogeneous holomorphic polynomials on $\C^{m+1}$ of increasing  degrees $N$.
The definition of ``quantum ergodic'' depends on the choice of a Hermitian metric $h$ on $L$,
and a probability measure $\nu$ on $M$. 
 In the positive Hermitian  line bundle 
case,  where the curvature  form $\omega_h =  i \ddbar \log h$ \footnote{Also written $\omega_{\phi} $ with $\phi = - \log h$.} of $h$ is the \kahler metric $\omega$, and where $d\nu $ is the volume form $dV_{\omega}= \frac{\omega^m}{m!} $ of $\omega$, a
 sequence $\{s_N\}$ with $s_N \in H^0(M,
L^N)$ of (not necessarily normalized) sections is {\it \kahler
quantum ergodic} if \begin{equation} \label{QE} \frac{|s_N(z)|^2_{h^N}}{||s_N||^2_{L^2}} dV_{\omega}\to
d V_{\omega}\; \rm{(weak*),\; i.e.}  \int_M f(z)\; \frac{|s_N(z)|^2_{h^N}}{||s_N||^2_{L^2}} 
dV_{\omega}  \to \int_M f\; dV_{\omega} \; (\forall f \in C(M)). \end{equation}
Properties of such sections were studied in the \kahler context in \cite{Z97}.
In \cite{NV98, ShZ99} it was shown that quantum ergodicity of a sequence implies that the normalized logarithms 
\begin{equation} \label{uNdef} u_N: = \frac{1}{N} \log \frac{|s_N(z)|^2_{h^N}}{||s_N||^2_{L^2}} \end{equation}
are  asymptotically
extremal (quasi-) plurisubharmonic functions, in the sense that 
$\limsup_N u_N \leq 0$ and $u_N \to 0$   in $L^1(M)$. Equivalently, if we express  $s_N = f_N e^N$ as a local holomorphic function $f_N$ relative to a
local holomorphic frame $e_L$ of $L$, and write \begin{equation} \label{eL} |e_L(z)|_h = e^{- \phi(z)} \end{equation} then 
\begin{equation} \label{fN} \frac{1}{N} \log \frac{|f_N(z)|^2}{||s_N||^2_{L^2}} \to \phi. \end{equation}
The potential $\phi =  -\log h$ of $\omega_h$ is the `equilibrium potential'  of $\omega_{\phi}$ in this positive line bundle  setting. Moreover, it was
shown  in \cite{ShZ99} that sequences of  random orthonormal bases of $H^0(M, L^N)$ 
are almost surely quantum ergodic. The proofs are based in part on the 
asymptotics of Bergman kernels of positive Hermitian line bundles. 

Over the last fifteen years, there has been a steady progression of generalizations of Bergman 
kernel asymptotics and asymptotics of random zero sets from the positive Hermitian line bundle case to general smooth metrics on ample (or just big) line bundles and Bernstein-Markov measures.  In particular, R. Berman initiated a new line of research with his articles \cite{Ber09, Ber09a} on Bergman kernels for pairs $(h, d\mu)$ where $d\mu$ is a volume form  and $h$ is a $C^2$ Hermitian metric on an ample line bundle. Later, in \cite{BBWN11}, the measure was allowed to be  any Bernstein-Markov measure $\nu$. 
We now generalize the definition and properties
of quantum ergodic sections to such $(h, \nu)$.

The definition involves the inner products ${\rm Hilb}_N(h, \nu)$
induced by the data $(h, \nu)$    on the spaces  $H^0(M,
L^N)$ of holomorphic sections of powers $L^N \to M$  by
\begin{equation} \label{HILB} ||s||^2_{{\rm Hilb}_N(h, \nu)} : = \int_M |s(z)|^2_{h^N} d\nu(z). \end{equation}
We let  $h $ \eqref{eL} be a general $C^{2}$ Hermitian metric on $L$, and
  denote its positivity set by  \begin{equation} \label{M0} M(0) = \{x \in M:  \omega_{\phi}  |_{T_x M}\; \rm{has \; only\; positive\; eigenvalues} \}, \end{equation}
i.e. the set where $\omega_{\phi}$  is a positive $(1,1)$ form. 
For a compact set $K \subset M$, also define  the  {\it equilibrium
potential }  $\phi_{eq} = V^*_{h, K}$ \footnote{Both notations $\phi_{eq} $ and $ V^*_{h, K}$, and also $P_K(\phi)$, are standard and we use them interchangeably. $V^*_{h,K}$ is called the pluri-complex Green's function in \cite{Bl05} and elsewhere.}
\begin{equation}\label{EQPOT} V_{h,
K}^* (z)= \phi_{eq}(z):= \sup \{u(z): u \in PSH(M, \omega_0), 
u \leq  \phi  \; \mbox{on}\;
 K\},
\end{equation}
where $\omega_0$ is a reference \kahler metric on $M$ and $ PSH(M, \omega_0)$ are the psh functions $u$ relative to $\omega_0$, i.e.(see \cite[Definition 2.1]{GZ05})
\begin{equation} \label{PSHM} PSH(M, \omega_0) = \{u \in L^1(M, \R \cup \infty): dd^c u + \omega_0  \geq 0,\;\; {\rm and}\; u \; {\rm is} \; \omega_0-u.s.c.\}. \end{equation} Further define the coincidence set,
\begin{equation} \label{D} D : = \{z \in M: \phi(z) = \phi_e(z)\}. \end{equation} Following Berman, we define the {\it equilibrium measure} associated to \eqref{eL} by 
\begin{equation}\label{EQMDEF} d\mu_{\phi} = (dd^c \phi_{eq})^m/m! =  {\bf 1}_{D \cap M(0)}   (dd^c \phi)^m/m!. \end{equation}
Here, $d^c = \frac{1}{i} (\partial - \dbar)$.
Finally, we fix a probability measure $\nu$  satisfying the Bernstein-Markov
property,
\begin{equation} \label{BM}  \sup_{z \in K} ||s(z)||_{h^N} \leq
C_{\epsilon} e^{\epsilon N} ||s||_{{\rm Hilb}(h^N, \nu)}, \;\; \forall \epsilon > 0, \;
\end{equation}
where as above $K =
\mbox{supp}\; \nu.$

The generalization of  \eqref{QE} is given in the following

\begin{defin}  \label{QEDEF}  Given $(h, \nu)$ as above, we say that $\{s_N\} $ with $s_N \in H^0(M, L^N)$ is
a quantum ergodic sequence with respect to $(h,\nu)$  if
$$\frac{|s_N(z)|^2_{h^N}}{||s_N||^2_{{\rm Hilb}_N(h, \nu)}} d\nu \to
d\mu_{\phi} $$ in the weak* sense
of measures.

\end{defin}

As explained in Section \ref{RANQE},  the definition of
quantum ergodic sequence originated in the study of eigenfunctions of
quantum maps in the setting of positive line bundles over \kahler manifolds. In generalizing the definition to the  $(h, \nu)$ setting of Definition \ref{QEDEF}, we cannot follow this approach since we do not currently have a definition of quantum map in the
general setting. It was later realized that quantum ergodic sequences
behave  like random ones in terms of their first two moments, so one may reverse the sequence of events and define quantum ergodic sequences as ones which have the same asymptotics of the first two as  random ones (see Section \ref{RANSEQINTRO}  for definitions
and Lemma \ref{EMASS} for the rigorous statement). By the easy calculation of \eqref{EXNcAs}, this boils down to 
$$\frac{|s_N(z)|^2_{h^N}}{||s_N||^2_{{\rm Hilb}_N(h, \nu)}} d\nu \simeq
N^{-m} \Pi_{h^N, \nu}(z) d\nu, $$
and by the Bergman kernel asymptotics of Berman, Witt-Nystrom and
others (see Theorem \ref{Ber1}), the right side tends to $d \mu_{\phi_{eq}}$.
Hence the Definition above is consistent with the comparison to random
sequences.

The first result is that normalized logarithms of ergodic sequences tend
to the equilibrium potential.

\begin{theo} \label{QETH}  Let $L \to M$ be an ample line bundle over
a projective \kahler manifold $M$. Let $h$ be a $C^2$ metric on $L$, and let $\nu$ be  a Bernstein-Markov measure.
If $\{s_N\}$ is a quantum ergodic sequence with respect to $(h, \nu)$,
then $u_N \to \phi_{eq}$ in $L^1(M, d V)$ (with respect to 
a  volume form $dV$  on $M$. \end{theo}
Here, $u_N$ is defined by \eqref{uNdef}.

\subsection{Zeros of a quantum ergodic sequence}

In the case of positive Hermitian line bundles, it was shown in \cite{ShZ99} that the normalized currents of integration over the zero sets $Z_{s_N}$
of a quantum ergodic sequence of sections tends to the \kahler form
$\omega$. This is simply a corollary of the fact that $u_N \to 0$ in $L^1(M)$. 
Indeed, the Poincar\'e-Lelong formula gives
\begin{equation} \label{PL} Z_s =\frac{\sqrt{-1}}{\pi} \ddbar \log|f| =
\frac{\sqrt{-1}}{\pi} \ddbar \log \|s\|_{h_n} + N\omega \;,\end{equation}
where locally $s= f e_L^N$ in a frame. Let $\wt Z_{s_N} = \frac{1}{N} Z_{s_N}.$  Then   for any smooth test form $\psi\in{\mathcal
D}^{m-1,m-1}(M)$, we have
$$\left(\wt Z_N-\omega,\psi\right)=\left(u_N,
\frac{\sqrt{-1}}{\pi}\ddbar \psi\right)\to 0\;,$$
and from  $$\left(\wt Z_N,\psi\right)\leq
\frac{c_1(L)^m}{m!}\sup|\psi|\;,$$ the conclusion of the lemma holds for
all $C^0$ test forms $\psi$.

  Theorem \ref{QETH} allows for a generalization of this result to the
  setting of this note and shows   that  the normalized zero currents of random sequences tend to
the {\it equilibrium metric} $dd^c \phi_{eq}$. Thus,

\begin{cor} \label{zerotheorem}   Let $L \to M$ be an ample line bundle over
a projective \kahler manifold $M$. Let $(h, \nu)$ be any pair as above.
If $\{s_N\}$ is a quantum ergodic sequence with respect to $(h, \nu)$, $\frac{1}{N}Z_{s_N}\to dd^c \phi_{eq}$ weakly in the sense
of measures. \end{cor}

\subsection{\label{RANSEQINTRO} Existence of quantum ergodic sequences}

Results on quantum ergodic sequences are  only useful if we can produce examples of quantum ergodic
sequences. In \cite{ShZ99}, B. Shiffman and the author proved that when
$\omega_h$ is a \kahler metric and $\nu = dV_{\omega}$, then  a
{\it random sequence}, and moreover a 
{\it random orthonormal basis }, of sections is quantum ergodic. The main result
of this note generalizes this statement to general smooth Hermitian
metrics $h$ and smooth volume forms $d\nu$. 

First, we define random sequence and random orthonormal basis. 
Each  inner product $\rm{Hilb}_N(h, \nu)$  induces a Gaussian measure
$\gamma_{h^N, \nu}$ on $H^0(M, L^N)$ and an associated spherical
measure $\mu_{h^N, \nu}$ on the unit spheres $SH^0(M, L^{N})$ in
$H^0(M, L^{N})$ with respect to ${\rm Hilb}_N(h, \nu)$ (see Section \ref{BACKGROUND}). We then have the
notion of a `random' sequence of ${\mathcal L}^2$-normalized
sections of $H^0(M, L^{N})$. Namely, we consider the probability
space $({\mathcal S}, d\mu)$, where \begin{equation} \label{SPHERICAL} {\mathcal S} =
\prod_{N=1}^{\infty} SH^0(M, L^{N}),\;\;\; \mu = \prod_{N =
1}^{\infty} \mu_{h^N, \nu}. \end{equation}  We refer to the elements of $\scal$
as random $L^2$-normalized sequences; see  \S \ref{BACKGROUND} for background.
Using results on the off-diagional asymptotics of the Bergman kernel 
for certain pairs $(h, \nu)$ of  R. Berman \cite{Ber09,Ber06}, we prove
\begin{theo} \label{RANSEQ} For any $C^2$ metric $h $ on $L$ and for any  smooth volume form $d\nu$,
almost every sequence $\{s_N\} \in \scal$   is quantum ergodic in the sense of Definition
\ref{QEDEF}.





 \end{theo}

In 
\cite{ShZ99}, the stronger result is proved   that random
orthonormal bases of sections, not just individual sections of
each degree, are quantum ergodic. 
As in \cite{ShZ99}, we let \begin{equation} \label{ONB} {\mathcal ONB} = \prod_{N = 1}^{\infty}
{\mathcal ONB}_N \end{equation} denote the product space of orthonormal bases of
$H^0(M, L^N)$ with respect to a given Hermitian inner product
${\rm Hilb}_N(h, \nu)$. Each ${\mathcal ONB}_N$ may be identified with
$U(d_N)$ where  $d_N + 1 = \dim H^0(M, L^N)$. We endow ${\mathcal ONB} $ with  the product of unit mass Haar measures.
An orthonormal basis will be denoted $\scal_N = \{S^N_0, \dots,
S^N_{d_N}\}$.

Using the same results on the off-diagional asymptotics of the Bergman kernel 
for certain pairs $(h, \nu)$, we prove

\begin{theo} \label{RANONB} Let  $h = e^{- \phi}$ be a smooth Hermitian metric on $L$ and
and let $d\nu$ be a smooth probability measure (normalized volume form). Then,  almost
every sequence $\{\scal_N\} $  of orthonormal bases of  $H^0(M,
L^N)$ is quantum ergodic in the sense of Definition \ref{QEDEF}.

 \end{theo}
 
 Both results would extend from smooth volume forms to general Bernstein-Markov measures if the off-diagonal Bergman kernel asymptotics of Theorem \ref{BERMANOFFD} below could be extended to that setting.

\subsection{\label{RANQE} Random sequences versus quantum ergodic sequences}

It follows from the results of this article that random sequences
are quantum ergodic and vice-versa that quantum ergodic sequences
behave in some ways like  
random sequences. However, 
there are more mechanisms to produce  quantum ergodic sequences than
just by taking random sequences.  For instance, quantum ergodic sequences
arise as eigensections of unitary  quantum ergodic maps in the positive line bundle case \cite{Z97}. Quantum maps are quantizations of symplectic maps of $(M, \omega)$. In the positive line bundle case, 
a (quantizable) symplectic map $\chi$ on $(M, \omega)$  is quantized as a unitary Toeplitz Fourier
integral operator of the form $U_{N, \chi}: = \Pi_N \sigma_N T_{\chi} \Pi_N$ where
$\Pi_N = \Pi_{h^N, dV_{\omega}}$ where $\omega_h = \omega$, where
$\sigma_N$ is a certain semi-classical symbol and where $T_{\chi} s(z)
= s(\chi(z))$ is the translation (or Koopman) operator corresponding to
$\chi$. In fact, to define $T_{\chi}$ on sections of $L^N$ it is necessary to lift the sections and $\chi$ to the associated principal $S^1$ bundle $X_h \to M$ . Alternatively one could parallel translate sections along paths from $z$ to $\chi(z)$ (see \cite{Z97,FT15}).

  It would be interesting to generalize this definition to the setting of
  this article, where $h$ is any smooth Hermitian metric and $d\nu$ is any
  Bernstein-Markov measure.  However, when  $dd^c \log h$ is not a symplectic form, it is not clear what is the appropriate generalization of  `symplectic map'  $\chi$ or of its quantization. In the true symplectic setting, $\chi^* \omega = \omega,$ (so that
$\chi^* dV_{\omega} = dV_{\omega}$). For general $(h, \nu)$, the natural
generalization would seem to be that $\chi^* dd^c \phi_{eq} = dd^c  \phi_{eq}$ (so that
$\chi^* \mu_{eq} = \mu_{eq}$). One may try to quantize $\chi$ by a formula
similar to the above and see if $U_{N, \chi}$   is an asymptotically  unitary operator modulo lower order terms on $H^0(M, L^N)$  with respect to ${\rm Hilb}_N(h, \nu)$ as $N \to \infty$.\footnote{In other words, is $U_N^* U_N = \Pi_N + o(1)$ where $o(1)$ is measured in the operator norm.} Other foundational questions are whether $U_{N, \chi}$ satisfies at least a weak
form of the Egorov theorem (see \cite{Ze98,FT15}), or whether the associated \szego kernel possesses scaling asymtpotics on the coincidence set \eqref{D}.
  One might hope to prove that its eigensections should be quantum ergodic if $\chi$ is ergodic 
with respect to $d\mu_{eq}$. At this time of writing, it is not obvious that
   the proposed quantization of such a map   produces  an asymptotically unitary operator,
 since only a shadow of  the  usual `symbol calculus' of Toeplitz Fourier integral 
operators \cite{BoGu81}  can  be  expected to generalize to this setting.   
      
\subsection{Acknowledgements} Thanks to R. Berman for comments
on an earlier version, in particular for emphasizing that Definition \ref{QEDEF} should be consistent
with the expected mass formula for random sequences.  Thanks also to T. Bayraktar for remarks and references on  Theorem \ref{BERMANOFFD},  and to the referees for many comments that helped improve the exposition.

\section{\label{BACKGROUND} Background}

We work throughout in the setting of \cite{ShZ99} and the more general one of \cite{BBWN11}. We let $M$ be a compact projective complex manifold of (complex) dimension $m$, and
let $L \to M$ be an ample holomorphic line bundle. The space of holomorphic
sections of the $N$th power of $L$ is denoted $H^0(M, L^N)$. In \cite{Ber09,BBWN11}, the line bundle is only assumed to be big, and that
is enough for most of the results, relying as they do on the results
of \cite{Ber06, Ber09,BBWN11}.  But in this survey we assume that $L$ is ample.

Let $h$ be a smooth (at least $C^2$) 
Hermitian metric on $L$ and denote its 
curvature form by
$$\label{curvature}\Theta_h=-\ddbar
\log\|e_L\|_h^2\;.$$  Here, $e_L$ is a local non-vanishing
holomorphic section of $L$ over an open set $U\subset M$, and
$\|e_L\|_h=h(e_L,e_L)^{1/2}$ is the $h$-norm of $e_L$ \cite{GH}. 
In the positive line bundle case it is assumed
that $$\omega_{\phi}=\frac{\sqrt{-1}}{\pi }\Theta_h$$ is a \kahler form. This is the assumption in \cite{ShZ99} but we do not  make that assumption in this note.
As in \cite{BerWN, BBWN11},  we consider the more general situation of
a holomorphic line bundle $L \to M$ together with 

\begin{itemize}

\item A $C^2$ Hermitian metric  $h = e^{- \phi}$ \eqref{eL};

\item A compact non-pluripolar set $K$ \footnote{A pluripolar set is a subset of the $-\infty$ set of a plurisubharmonic function.} and a  stably Bernstein-Markov measure $\nu$  with respect to $(K,
\phi)$ \eqref{BM}.

\end{itemize}
As above,
 we denote the set where the form
is \kahler by $M(0)$ \eqref{M0}.

Given a compact non-pluripolar set $K \subset M$, the equilibrium potential $\phi_{eq} $ is defined  as the upper semi-continuous
regularization of the  upper envelope   \eqref{EQPOT}.  

The associated equilibrium measure is the Monge-Ampere measure
${\rm MA}(\phi_{eq})$
of $\phi_{eq}$, defined by \begin{equation} \label{MADEF} {\rm MA}(\phi):=  (dd^c \phi_{eq})^m/m!. \end{equation}

\subsection{\szego kernel}

 The data $(h, \nu)$ induces the inner product \eqref{HILB} on $H^0(M, L^N)$.  The corresponding orthogonal
projection is then denoted by $$\Pi_{h^N, \nu}: L^2(M, L^N) \to H^0(M, L^N), $$
where the inner product is given by  ${\rm Hilb}_N(\phi, \nu)$.
If  $\{S_j^N\}$ of $H^0(M, L^N)$ is an orthonormal basis  with
respect to the inner product ${\rm Hilb}_N(\phi, \nu)$, then  the Schwartz kernel of $\Pi_{h^N, \nu}$ with respect to $d\nu(z)$ is given by,
$$\Pi_{h^N, \nu}(z,w) = \sum_{j = 1}^{d_N} S^N_j(z) \otimes \overline{S^N_j(w) }$$
in the sense that 
$$\Pi_{h^N, \nu} s (z)  = \int_M \langle \Pi_{h^N, \nu}(z,w),
s(w)\rangle_{h^N} d \nu(z). $$ 

We denote
the {\it density of states}  by \begin{equation}
\label{DOS} \Pi_{h^N, \nu}(z):   ;=  \sum_{j = 1}^{d_N} |S^N_j(z)|^2_{h^N}. \end{equation}
If we write $S_j^N = f_j^N e_L^N$ in a local frame, then we also define
the Bergman kernel by
\begin{equation} \label{BKdef}  B_{h^N, \nu} (z,w) =  \sum_{j = 1}^{d_N} f_j^Nj(z) \overline{f_j^N(w)}, \end{equation}
so that $ \Pi_{h^N, \nu}(z,w) = B_{h^N, \nu}(z,w) e_L^N(z) \overline{e_L^N(w)}. $ We also define the 
 Bergman measure by\footnote{The notation $B_{h^N, \nu}(z,z)$ is used in articles of Berman; $\Pi_{h^N, \nu}(z)$ is the contraction of the diagonal $\Pi_{h^N, \nu}(z,z)$.}
$$\Pi_{h^N, \nu}(z) d\nu = B_{h^N, \nu} (z,z) e^{- N \phi} d\nu. $$

For $N$ sufficiently large, $B_{h^N, \nu} (z,z)$ is everywhere positive
in the case of an ample line bundle, since there is no point $z$
where all sections vanish (the base locus). But  the Bernstein-Markov measure, hence the Bergman measure,  is supported 
on $K := {\rm supp}\; \nu$.

\subsection{\label{GFS}  Spherical and Gaussian measuress on $H^0(C, L^N)$ induced by Hermitian inner products}

The data  $(h, \nu)$ induces the inner product \eqref{HILB}. 
Let $d_N = \dim H^0(M, L^N)$.
The inner product induces  a Gaussian
measure $\gamma_N = \gamma_N(h, \nu)$  on this complex vector space by the formula,
\begin{equation}\label{gaussian1}d\gamma_N(s_N):=\frac{1}{\pi^m}e^
{-|c|^2}dc\,,\quad s_N=\sum_{j=1}^{d_N}c_jS^N_j\,,\quad
c=(c_1,\dots,c_{d_N})\in\C^{d_N}\,,\end{equation} where
$\{S_1^N,\dots,S_{d_N}^N\}$ is an orthonormal basis for
$H^0(M, L^N)$, and $dc$ denotes $2d_N$-dimensional
Lebesgue measure.   The measure $\ga_N$ is characterized by the
property that the $2d_N$ real variables $\Re c_j, \Im c_j$
($j=1,\dots,d_N$) are independent Gaussian random variables with
mean 0 and variance $1/2$; equivalently,
$$\E_N c_j = 0,\quad \E_N c_j c_k = 0,\quad  \E_N c_j \bar c_k =
\de_{jk}\,,$$ where $\E_N$ denotes the expectation with respect to
the measure $ \ga_N$.

  We also define the spherical measure  $d\mu_{h^N, \nu}$ to  be the unit
mass Haar measure on $SH^0(M, L^N)$, the sections of $L^2$ norm 1. 
 The spherical measure is  equivalent to the
 {\it
normalized  Gaussian measure}
\begin{equation}\label{gaussian}\mu_N:=\tilde\ga_{2d_N}=
(\frac{d_N}{\pi})^{d_N} \;e^ {-d_N|c|^2}d\lcal(c)\,,\quad\quad
s=\sum_{j=1}^{d_N}c_jS_j^N = \langle \vec c, \vec S \rangle\,,\end{equation} where $\{S_j^N\}$ is
an orthonormal basis for $H^0(M,L^N)$ and $d\lcal(c)$ is Lebesgue measure on the $\R^{2m} \simeq \C^m$ .  Recall that this
Gaussian is characterized by the property that the $2d_N$ real
variables $\Re c_j, \Im c_j$ ($j=1,\dots,d_N$) are independent
random variables with mean 0 and variance $1/2d_N$; i.e.,
\begin{equation}\label{normalized}\langle c_j \rangle_{\mu_N}= 0,\quad
\langle c_j c_k\rangle_{\mu_N} = 0,\quad \langle c_j \bar c_k
\rangle_{\mu_N}= \frac{1}{d_N}\de_{jk}\,.\end{equation} 

\subsection{Expected mass}
\begin{lem} \label{EMASS} For either the spherical ensemble or the normalized Gaussian
ensemble,
$$\E ||s_N(z)||^2 = \frac{1}{d_N} \Pi_{h^N, \nu}(z). $$
\end{lem}
\begin{proof} 
It is obvious that
$$\E \left| \sum_{j=1}^{d_N}c_jS^N_j \right|_{h^N}^2
=\frac{1}{d_N} \sum_j |S^N_j(z)|_{h^N}^2 = \frac{1}{d_N} \Pi_{h^N, \nu}(z). $$
Note that $d_N = N^m (1 + O(\frac{1}{N})$.
\end{proof}

\subsection{Expected distribution of zeros of  random polynomials}

Let $M  = \CP^1$, let $L = \ocal(1)$ (the dual of the hyperplane line bundle \cite{GH}) and let $\ocal(N) = L^N$.  Then $H^0(\CP^1, \ocal(N))$ may be identified with the space $\pcal_N$ of polynomials $p_N$ on $\C$ of degree $N$. That is, in a frame $e_L$ over the affine chart $\C $, $s_N = p_N e^L$ (see \cite{GH} or \cite{ShZ99} for this standard fact). 

The empirical measure  of zeros $\{\zeta_1, \dots, \zeta_N\}$ of $p_N \in \pcal_N$ is the
probability measure on $\C$ defined by
\begin{equation} \label{EMP}\frac{1}{N} [Z_{p_N}] =  \mu_{\zeta} =  \frac{1}{N} \sum_{\zeta_j : p_N(\zeta_j) = 0} \delta_{\zeta_j}. \end{equation}

\begin{defin} For any probability measure $P$ on $\pcal_N$, the expected distribution of zeros   of $p_N \in \pcal_N$ is the probability measure $\E_N
\frac{1}{N} Z_{p_N}$ on $\C$ defined on a test function $\phi \in C_c(\C)$ by
$$\langle  \E_N \;\mu_{\vec \zeta}, \phi \rangle = \int_{\pcal_N} \{\frac{1}{N}  \sum_{\zeta_j : p_N(\zeta_j) = 0} 
\phi(\zeta_j) \} dP(p_N), $$  $$
 =  \frac{1}{N}  \frac{i}{2 \pi}  \int_{\pcal_N}\left(\int_{\C} \phi \ddbar \log |p_N| \right) dP(p_N), $$ 
\end{defin}
Recall that in complex dimension one, if $f(z)$ is a complex analytic
function, then by \eqref{PL},
\begin{equation}
\label{PL2}[ Z_f]  = \sum_j \delta_{\zeta_j} =  \frac{i}{2 \pi} \ddbar \log |f|^2  = \frac{i}{2\pi}
\frac{\partial^2 \log |f|^2}{\partial z \partial \bar{z}} dz \wedge d \bar{z}. \end{equation}

The definition extends with no essential change to ample line bundles $L \to M$ over a Riemann surface $C$, except that the number of zeros is the degree  $N c_1(L)$ of $L^N$. We assume for simplicity that $c_1(L)  = 1$.
  For $s\in H^0(C, L^N)$, we let $Z_s$
  denote empirical measure of zeros,
$$\frac{1}{N} ([Z_s],\psi)= \frac{1}{N}  \sum_{z: s(z) = 0} \psi(z)$$  When
$s= f e_L^{\otimes N}$, we have by the \Poincare-Lelong formula \eqref{PL},
\begin{equation}
\label{PL3} \frac{1}{N} [Z_s]  = \frac{i}{N \pi} \ddbar \log |f| = \frac{i}{N\pi}
\ddbar \log \|s\|_{h^N} + \omega_{\phi}\;.\end{equation}

In  higher dimensions,  the zero set $Z_{s_N}$
is a complex hypersurface rather than a discrete set of points. 
For a general ample line bundle over a \kahler manifold, we also have:

\begin{lem} Let $\{s_j^N\}$ be an orthonormal basis of $H^0(M, L^N)$. 
Let $s_j^N =f_j^N e^N$.  Then,
$$\E_N (\frac{1}{N} [Z_s^N]) = \frac{\sqrt{-1}}{2 \pi N} \ddbar \log \sum_{j=1}^{d_N}
|f^N_j|^2. $$ Moreover,  $$ \E_N (\frac{1}{N} [Z_s^N]) = \frac{\sqrt{-1}}{2 \pi N} \ddbar \log
B_{h^N, \nu}(z,z) =  \frac{\sqrt{-1}}{2 \pi N} \ddbar \log
\Pi_{h^N, \nu}(z,z) + \omega_{\phi}.
$$

\end{lem}

The proof of the Lemma is simple. 
Let $s = \sum_j a_j s_j^N$ and write it as $\langle \vec a, \vec s^N \rangle
= \langle \vec a, \vec f^N \rangle e^N$. 
 Let $\psi \in C^2(M)$. Then
 $$ \E_N \langle \frac{1}{N} [Z_s^N]), \psi \rangle = \frac{\sqrt{-1}}{\pi N} \int_{\C^{d_N}} \int_M \ddbar \log
|\langle \vec a, \vec f\rangle|  \psi d\gamma_N(a) \;.$$ To compute the integral, we write $\vec f = |\vec f|
\vec u$ where $|\vec u| \equiv 1.$ Evidently, $\log |\langle \vec a, \vec f\rangle| =
\log |\vec f| + \log |\langle \vec a, \vec u \rangle|$. The first term gives
\begin{equation} \frac{\sqrt{-1}}{\pi N} \int_M \ddbar \log
|\vec f|  \psi = \frac{\sqrt{-1}}{\pi N} \int_M \ddbar \log \Pi_{h^N}(z,z)
  \psi  +  \int_M \omega_{\phi} \psi. \end{equation}

 We
now look at the second term.  We have $$\label{term=0}
\frac{\sqrt{-1}}{\pi} \int_{H^0(C, L^N)} \int_M \ddbar \log
|\langle \vec a, \vec u\rangle|  \psi d\gamma_N(a)$$ $$=
\frac{\sqrt{-1}}{ \pi}\int_M  \ddbar \left[ \int_{H^0(M, L^N)}
\log |\langle \vec a, \vec u \rangle|  d\gamma_N(a)\right] \psi =0,$$
since the average $\int \log |\langle \vec a, \vec u \rangle|
d\gamma_N(a)$ is a constant independent of $\vec u$ for $|\vec u|=1$, and thus
the operator $ \ddbar$ kills it.

When $\omega_{\phi}$ is a \kahler form and $\nu = \frac{\omega_{\phi}^m}{m!}$, there exists a complete asymptotic expansion, 
$$\Pi_{h^N, \nu}(z,z) =  a_0 N + a_1(z) + a_2(z) N^{-1} + \cdots$$
and it follows that
$$\E_N (\frac{1}{N} [Z_s^N])  \to \omega_{\phi}. $$
We refer to \cite{Ze98,ShZ99} for background.

We now turn to the case where the curvature of $h$ is not necessarily postive.

\section{\label{NONSTANDARD} Non-standard Bergman kernel asymptotics}

The proof of Theorem \ref{QETH}  relies on several results of R. Berman \cite{Ber06, Ber09,Ber09a},
of Berman-Witt-Nystrom \cite{BerWN} and Berman-Boucksom-Witt-Nystrom \cite{BBWN11}.   Let $h$ be a $C^2$
Hermitian metric, 
locally defined in a holomorphic frame \eqref{eL}.

The results below do not use parametrix constructions for the Bergman kernel (see \cite{BoSj,BBSj,Ze98} for background on parametrices). The starting point is the extremal property
\begin{equation} \label{EXTREMAL} 
\Pi_{h^N,  \nu}(z) = \sup \{|s_N(z)|^2_{h^N} \; s_N \in H^0(M, L^N), \;\;
||s_N||_{{\rm Hilb}_N(h, \nu)} = 1. \}. \end{equation}
This is an immediate consequence of the fact that  $\Pi_{h^N,  \nu}$ is the orthogonal projection with respect to ${\rm Hilb}_N(h, \nu)$, i.e. that
$s_N = \int_M (\Pi_{h^N,  \nu}, s_N)_{h^N} d \nu$. The Cauchy-Schwartz inequaltiy gives the upper bound. The lower bound follows by using 
$\frac{\Pi_{h^N,  \nu}(\cdot, z)}{||\Pi_{h^N,  \nu}(\cdot, z)||_{{\rm Hilb}_N(h, \nu)}} $ for $s_N$.

 Theorem 1.3 of \cite{Ber09}  states the following asymptotics of the density
 of states \eqref{DOS}:
\begin{theo} \label{Ber1} Let $L \to M$ be an ample line bundle over a \kahler manifold and let $h$ be a $C^2$ Hermitian metric on $L$. If $\nu $ is a smooth volume form, then in the weak* sense of measures,
\begin{equation}  N^{-m} \Pi_{h^N, \nu} (z) \nu \to d\mu_{\phi} = {\bf 1}_{D \cap M(0)}  \; (dd^c \phi)^m(z)/m!= {\rm MA} (\phi_{eq}), \end{equation}
the equilibrium measure \eqref{EQMDEF}. Moreover, $$N^{-m} \Pi_{h^N, \nu}(z) \to {\bf 1}_{M(0) \cap D} \; \det (dd^c \phi_{eq}(z))$$
almost everywhere, where $\det (dd^c \phi_{eq}(z)) =  \frac{(dd^c \phi_{eq}(z))^m}{m!\; d\nu(z)}. $
\end{theo}

Theorem B of \cite{BBWN11} gives a more general result.   

\begin{theo} \label{BBWNTH} Let $(X, L)$ be a compact complex manifold equipped with an ample line bundle $L \to X$. Let $K$ be a non-pluripolar compact subset of $X$ and $\phi$ a continuous weight on $L$. Let $\mu$ be a Bernstein-Markov measure for $(K, \phi)$. Then,
\begin{equation} \label{MALIM2} N^{-m} \Pi_{N \phi}(z) d\nu \to {\rm MA}(\phi_{eq}) \; \rm{in\; the\; weak*\; sense}.
\end{equation}
\end{theo}
It may be useful to state the result in the notation of \cite{BBWN11}:  They denote the  normalized density of states \eqref{DOS} by
$$\beta(\nu, \phi) = \frac{1}{N} \rho(\nu, \phi) \nu, \;\rm{where}\;\rho(\nu, \phi) = \sum |s_i|^2_{h^N}. $$  Theorem B of \cite{BBWN11} then states,
$$\lim_{N \to \infty} \beta(\nu, N \phi) = \mu_{\phi}. $$

The non-standard Bergman asymptotics are illustrated in \cite{ShZ03} in
the case where $h = 1$ and with simply connected  analytic plane domains in terms of exterior Riemann mapping functions.

Let us briefly indicate some key ideas in the proof of Theorem \ref{BBWNTH}. In the setting of general weights and measures, it is non-standard to construct parametrices for the Bergman kernel as used in
\cite{BoSj,BBSj,Ze98}. Instead, in  \cite{BerWN, BBWN11}, the
 proofs depend on the fact that the Bergman measure (density of states) is the differential of a certain functional $F_N$ on the affine
space of all continuous weights (for a fixed compact set $K$). The $F_N$
converge to a concave functional $F$ with continuous \Frechet \;differential,
and the differential of $F$ is represented by the equilibrium measure
$\rm{MA}(\phi_{eq})$. Moreover,  $F_N$ is concave for any $N$. It follows
that the derivatives also converge. 
The functionals are defined by \footnote{$F_N$ is denoted $\lcal_N$ in \cite{BBWN11}.}
$$F_N(h, \nu) = \frac{(m+1)!}{2 N^{m+1}} \log \rm{Vol} \;\bcal^2(h, \nu). $$
where $\bcal^2(h, \nu)$ is the unit ball in $H^0(M, L^N)$ with respect to
${\rm Hilb}(h, \nu)$,
respectively,
$$F(h, \nu) = \ecal_0(\phi_{eq}), $$
where $\ecal_0 $ is the Monge-Ampere energy  functional (see \cite[Section2]{BBWN11}. Namely, $\ecal_0$ is the functional whose variational derivative
at the potential $\phi$ is represented by the Monge-Ampere measure ${\rm MA}(\phi)$).
Under a certain Bernstein-Markov condition, it is proved in \cite{BB10,BBWN11} that $$F(h, \nu) = \lim_{N \to \infty} F_N(h, \nu). $$
See \cite[Therem A]{BB10} and  \cite[(0.9)]{BBWN11} together with the Bernstein-Markov condition in \cite[p. 8]{BBWN11}.

The next series of results pertain to  normalized logarithms of Bergman kernels and their convergence to the equilibrium potential. 
The following asymptotics of the `Bergman metrics' is a combination of results found in  \cite[Theorem 3.7]{Ber09a},    \cite[Lemma 3.4]{BlS07}
and in  \cite[Lemma 2.3]{Bl07} and \cite[Proposition 3.1]{BL15}  for polynomials on $\C^m$ and Bernstein-Markov measures. See also    (1.9)  of \cite[Theorem 1.5]{Ber09} when
$(h, \nu)$ consists of a  smooth Hermitian metric on an ample line bundle,
and a smooth volume form and \cite[Proposition 2.9]{Bay16} for the statement in  the case of 
smooth Hermitian metrics on line bundles  and Bernstein-Markov measures,
but the proof is cited from \cite{Ber09, Ber09a}. 

 As above, write 
 $\Pi_{h^N, \nu} (z,w) = B_{h^N, \nu}(z,w) e_L^N(z) \overline{e_L^N(w)}$ in a local frame \eqref{BKdef}, and let $  B_{h^N, \nu} (z) =  \sum_{j = 1}^{d_N} |f_j^Nj(z) |^2. $ 

\begin{theo} \label{LOGBERG} For smooth Hermitian metrics $h$ on an ample line bundle $L \to M$ as above,  and
Bernstein-Markov measures $\nu$, 
\begin{equation} \label{EPLIM} \frac{1}{N} \log B_{h^N, \nu}(z) \to   \phi_{eq}(z) \end{equation} 
uniformly. 
\end{theo}

\begin{proof} (Sketch of proof following \cite{BlS07}) In the notation of \cite{BlS07,Bl07} let $K  = \rm{supp}\; \nu$ and let
$$\Phi_N^K(z) = \sup \{|s_N(z)|^2_{h^N}, \;\; s \in H^0(M, L^N),\; \sup_K
|s_N(z)|^2_{h^N} \leq 1\}. $$ 
This extremal function is almost the same as the density of states 
\eqref{EXTREMAL} except that the normalizing condition uses the 
sup norm on $K$  rather than the $L^2$ norm with respect to
${\rm Hilb}_N(h, \nu)$. By the Bernstein-Markov property of $\nu$, these 
two normalizations are asymptotically equivalent if one takes logarithms. 
Indeed, for any $\epsilon > 0$ there exists $C_{\epsilon} >0$ so that
\begin{equation} \label{BLOOM} \frac{1}{d_N} \leq \frac{\Pi_{h^N,\nu}(z)}{\Phi_N^K(z)} \leq C_{\epsilon}
e^{\epsilon N} d_N. \end{equation} Indeed, if $\sup_K |s_N(z)|^2_{h^N} \leq 1$, then
$$\begin{array}{lll} |s_N(z)|_{h^N} & = &\left| \int_K \Pi_{h^N, \nu}(z,w) \cdot s_N(w) d \nu(w) \right| \\ &&\\
& \leq & \int_K |\Pi_{h^N,\nu}(z,w)|) d\nu(w) \\ &&\\
&\leq & \int_K \Pi_{h^N,\nu}(z,z)^{\half}  \Pi_{h^N,\nu}(w,w)^{\half}d\nu(w)
\\&& \\
&= &  \Pi_{h^N,\nu}(z,z)^{\half} \nu(K)^{\half} [\int_K  \Pi_{h^N,\nu}(w) d\nu(w) ]^{\half} =  \Pi_{h^N,\nu}(z,z)^{\half}  d_N^{\half}. \end{array}$$
This inequality implies the left inequality of \eqref{BLOOM}. 

For the right inequality, one uses the Bernstein-Markov inequality
$\sup_K |S^N_j(z)|_{h^N} \leq C e^{\epsilon N} $ on an orthonormal 
basis $\{S^N_j\}$. By the definition of $\Phi_N^K$ one has
$ |S^N_j(z)|^2_{h^N} \leq C e^{\epsilon N} \Phi_N^K(z), $ so that
$$\Pi_{h^N, \nu}(z)  \leq d_N C e^{ \epsilon N} \Phi_N^K(z). $$
This completes the proof of \eqref{BLOOM}. It is clear that this estimate
is universal, i.e. does not use any special properties of $(h, \nu)$.

From \eqref{BLOOM} it follows that
$$\frac{1}{N} \log \frac{\Pi_{h^N, \nu}(z)}{\Phi_N^K(z)} \to 0. $$
This reduces the problem to finding the limit of  $\frac{1}{N} \log \Phi_N^K(z). $ It is immediate from the upper envelope definition \eqref{EQPOT}   of the equilibrium potential that $\frac{1}{N} \log \Phi_N^K \leq \phi_{eq}$ for
all $N$. 

For the reverse inequality $\lim_{N \to \infty} \frac{1}{N} \log \Phi_N^K \geq V_K$ one needs a generalization of the
Siciak-Zaharjuta theorem that $V_K(z) = \sup_N \frac{1}{N} \log \Phi_N^K(z). $ This requires the construction of sections $s_N$ saturating the lower bound asymptotically. In \cite[Theorem 3.7]{Ber09a}, which takes a somewhat different route, the lower bound is proved using the Ohsawa-Takegoshi extension theorem (see \cite[Lemma 5.2]{Ber09a} to construct global sections $s_N$ which satisfy the desired lower bound at one point. 

In \cite[Theorem 6.2]{GZ05}, Guedj-Zeriahi prove a Siciak-Zahajuta theorem on line bundles
over \kahler manifolds which is valid for any continuous weight (see also
\cite[Proposition 2.9]{Bay16}). \footnote{Thanks to Turgay Bayraktar for the reference and explanations of
this point.}

\end{proof}

  It follows (see   Theorem 1.4  of \cite{Ber09}) that one has:

 \begin{theo}  \begin{equation} \label{MALIM} \left( dd^c \log
|B_{h^N}(z,z)| \right)^m/m! \to \mu_{\phi}. \end{equation}

\end{theo}

We also need the following off-diagonal asymptotics of the Bergman kernel
given in 
 Theorem 1.7  of \cite{Ber09} (see also Theorem 2.4 of \cite{Ber06} where
 an additional assumption is made).

\begin{theo} \label{BERMANOFFD}   Let $L \to M$ be an ample line bundle, let $h$ be a smooth Hermitian metric and let $\omega_N$ be a smooth volume form,

\begin{equation} N^{-m}  | B_{h^N}(z,w)|^2 \omega_N(z) \wedge \omega_N(w) \to
\Delta \wedge d\mu_{\phi}.  \end{equation}

\end{theo}

In the case of postive Hermitian line bundles, $\Pi_{h^N}(z,w)$ decays rapdily off the diagonal.  The above gives a weak generalization to more general Hermitian metrics. It also gives the second moment part of the `Szego limit
theorem' due in the positive Hermitian case to Boutet de Monvel and Guillemin 
\cite{BoGu81}. It is not clear to the author whether Theorem \ref{BERMANOFFD} has been, or can be, generalized to $(h, \nu)$ where
$\nu$ is only assumed to be Bernstein-Markov.

\section{Proof of  Theorem \ref{QETH}}

\begin{prop} \label{QEPROP} Let $(L,h) \rightarrow (M, \omega)$ be a 
Hermitian holomorphic line bundle over a \kahler manifold $M$, with $h \in C^2$ and let $\nu$ be a Bernstein-Markov measure. Let $s_N \in H^0(M, L^{N})$, $N=1,2,\ldots$, be
a quantum ergodic sequence of  sections in the sense of Definition \ref{QEDEF}.  Then $\frac{1}{N} \log
||s_N||_{h^N} \to \phi_{eq} $ in $L^1(M,dV)$ where $dV$ is any fixed
volume form.
\label{W*}\label{mainlemma}\end{prop}

\begin{proof}

Let $s_N \in H^0(M, L^{N})$, $N=1,2,\ldots$ be quantum ergodic. We write
$$u_N(z) =\frac{1}{N} \log \frac{|s_{N}(z)|_{h^N}}{||s_N||_{{\rm Hilb}_N(h, \nu)}}\;.$$
Henceforth we assume $||s_N||_{{\rm Hilb}_N(h, \nu) }= 1$.
Let $e_L$ be a local holomorphic frame for $L$ over $U \subset M$
and let $e_L^{N}$ be the corresponding frame for $L^{N}.$ Let
$\phi (z) =- \log \|e_L(z)\|_{h}$ so that $\|e_L^{N}(z)\|_{h_N} = e^{-
N \phi} $. Then we may write $s_N = f_N e_L^{N}$ with $f_N \in
{\mathcal O}(U)$ and $\|s_N\|_{h_N} = |f_N| e^{- N \phi}$. It is equivalent 
to  consider the locally defined psh functions  $$v_N = \frac{1}{N}\log |f_N| =
u_N + \phi\;,$$ on $U$.

We wish to show that $u_N\to \phi_{eq}$ in $L^1(M)$. 
To prove this we first observe that $\{u_N\}$ is a pre-compact
sequence of quasi-psh (pluri-subharmonic) functions in $L^1$. The sequence  satisfies:
\begin{itemize} \item[\rm i)] the functions $u_N$ are
uniformly bounded above on $M$; \item[\rm ii)] \ $\limsup_{N
\rightarrow \infty} u_N\leq 0$. \item[\rm iii)] $u_N$ do not tend
to $-\infty$ uniformly on $M$.
\end{itemize} To prove (iii) we note that since $\|s_N\|_{h^N}^2d\nu$ converges weakly to ${\bf 1}_{D \cap M(0)}  (dd^c \phi)^m(z)/m!$, we have
\begin{equation} \label{MASS} \int_M\|s_N(z) \|^2_{h^N}d\nu \to \int_M {\bf 1}_{D \cap M(0)}  \rm{MA} (\phi)(z).\end{equation} If (iii) were to hold, the left side would tend to $0$. Indeed,  there would exist $T > 0$ such that for $k \geq T$,
\begin{equation}\frac{1}{N_k} \log \| s^{N_k}(z)\|_{h_{N_k}} \leq -
1.\label{firstposs}\end{equation} However, (\ref{firstposs}) implies
that $$ \| s_{N_k}(z)\|_{h^{N_k}}^2 \leq e^{- 2
N_k}\;\;\;\;\forall z \in U'\;, $$ which is inconsistent with \eqref{MASS}.

To prove (i) - (ii), choose orthonormal
bases $\{S^N_j\}$ and write $s_N=\sum_j a_j S^N_j$, so that $\sum
|a_j|^2 = ||s_N||_{L^2}^2$. By (\ref{BM}), we have $$\| s_N (z)\|_{h^N}^2
\leq C_{\epsilon} e^{ \epsilon N} ||s_N||^2_{{\rm Hilb}_N(h, \nu)}.$$
Taking the logarithm gives (i)-(ii) and therefore (i) - (iii).

Since $|s_N(z)|^2_{h^N} \leq \Pi_{h^N, \nu}(z)$ it makes sense to study the ratio
as a scalar function,
\begin{equation} \label{LEQ1} \frac{|s_N(z)|^2_h}{  \Pi_{h^N, \nu}(z)} \leq 1. \end{equation}
Moreover, it is the Radon-Nikodym derivative of $|s_N(z)|^2_h d \nu$
with respect to $\Pi_{h^N, \nu} (z) d\nu.$

The final step in the proof of Proposition \ref{QEPROP} is the following

\begin{lem} \label{RAT} If $\{s_N\}$ is a quantum ergodic sequence, then  
\begin{equation} \label{LOGRATIO} \frac{1}{N}\log \frac{|s_N(z)|^2_h}{  \Pi_{h^N, \nu}(z)} \to 0 \end{equation} in $L^1(M, dV)$.
\end{lem}

Let
$U'$ be a relatively compact, open subset of $U$.
By (i)-(iii), both  $\frac{1}{N}\log |S_j(z)|^2_h$ and $\frac{1}{N}\log  \Pi_{h^N}(z)$ are precompact in $L^1(M)$, and in the latter case the limit
is given in Theorem \ref{LOGBERG}. 
Thus,  it follows by  a standard result on subharmonic functions that   a sequence $\{v_{N_k}\}$
satisfying (i)-(iii)  has a subsequence which is convergent in
$L^1(U')$.  We choose a subsequence of
indices $(N, j)$ so that a unique $L^1$ limit of the log-ratio \eqref{LOGRATIO} exists, and prove that it
must equal $0$.
By \eqref{LEQ1}, any limit of  \eqref{LOGRATIO}  must be $\leq 0$.

We denote the log ratio by $$w_N = v_N -  \frac{1}{N} \log \Pi_{h^N, \nu}(z). $$ Then there  exists a subsequence  $\{w_{N_k}\}$, which converges in $L^1(U')$ to some $w
\in L^1(U').$  By passing if necessary to a further subsequence,
we may assume that $\{w_{N_k} \}$ converges pointwise almost
everywhere  in $U'$ to $w$, and hence
$$w(z) = \limsup_{k \rightarrow \infty} [u_{N_k}(z) + \phi -  \frac{1}{N_k} \log \Pi_{h^{N_k}, \nu}(z)]  \;\;\;\;\;\;{\rm (a.e)}\;.$$ Now let
$$v^*(z):= \limsup_{w \rightarrow z} v(w)  $$ be the
upper-semicontinuous regularization of $v$. Then $v^*$ is
plurisubharmonic on $U'$ and $v^* = v$ almost everywhere.

Assuming that $w \not= 0$,  there exists
$\epsilon
> 0$, so that the open set $ U_\epsilon = \{z\in U': w^* <  - \epsilon\} $ is
non-empty.
 Let $U''$ be a non-empty, relatively compact, open
subset of $ U_\epsilon$; by Hartogs' Lemma, there exists a
positive integer $T$ such that 
 \begin{equation}\label{ODDBOUND}
\|s_{N_k}(z)\|_{h^{N_k}}^2 \leq e^{- \epsilon N_k } N_k^{-m} \Pi_{h^{N_k}, \nu}(z),\;\;\;\;\;z\in
U'', \;\;k \geq T.
\end{equation}

By Theorem \ref{Ber1} and Theorem \ref{BBWNTH},   $N^{-m} \Pi_{h^N, \nu} (z) d\nu  \to d\mu_{\phi} = {\bf 1}_{D \cap M(0)}   (dd^c \phi_{eq})^m(z)/m!$ weak*.  Applying this to the right side of \eqref{ODDBOUND}  contradicts the weak convergence of the left side to $ {\bf 1}_{D \cap M(0)}  \det (dd^c \phi)(z)$ and completes the proof of  Lemma \ref{RAT}.

To complete the proof of Proposition \ref{QEPROP}, we observe that by Lemma \ref{RAT} and Theorem \ref {LOGBERG},
$$\begin{array}{lll} u_N(z) =\frac{1}{N}\log |s_N(z)|^2_{h^N} & = &  \frac{1}{N} \log  \Pi_{h^N, \nu}(z) + o(1) \\ &&\\
& = & \phi_{eq} + o(1), \end{array}$$
where the remainder is measured in $L^1(M, dV)$, proving 
 Proposition \ref{QEPROP} and Theorem \ref{QETH}.

\end{proof}

\section{Quantum ergodicity of random sequences: Proof of Theorem \ref{RANSEQ}}

There are several different ways to formulate the statement that random 
$L^2$ normalized sequences $\{s_N\}$ are quantum ergodic. In this
section we  prove the result using the Kolmogorov strong law of large numbers. This gives a weaker result than in the positive Hermitian line
bundle case of \cite{ShZ99}, where the variance estimate is good enough to allow us to 
apply Borel-Cantelli to prove almost sure convergence. But the approach here does  not seem to have  been used before, and so we present it here. It is quite close to the study of variances of linear statistics
of zeros in \cite{ShZ99,ShZ08}. The result is somewhat weaker than for random 
orthonormal bases of sections but the details of the proof are somewhat different although in the end the key point is to study certain quantum variances.  

\subsection{Proof of almost sure quantum ergodicity of sequences}

In this section we consider random sequences $\{s_N\} \in \scal$ in
the spherical model \eqref{SPHERICAL}. We identify a random section
in $H^0(M, L^N)$
with the associated coefficients $\vec c$ relative to a fixed orthonormal basis. 

In the notation \eqref{gaussian},  for $a \in C^{\infty}(M)$ we consider the random variables,
\begin{equation} \label{XNc} X_N^a(\vec c):=  \int_M a(z) |\langle \vec c, \vec S^N(z) \rangle|^2 d \nu(z). \end{equation}
Quantum ergodicity has to do with the variances,
\begin{equation} \label{VARXN} {\rm Var}(X_N^a) = \E_N \left| X_N^a(\vec c) - \int_M a \; {\rm MA}(\phi_{eq}) \right|^2, \end{equation}
(see \eqref{MADEF} for the notation).

By Lemma \ref{EMASS}, the  expected value is given by
\begin{equation} \label{EXNc} \E_N
X_N^a(\vec c)= \int_M a(z) \Pi_{h^N, \nu} (z) d\nu(z) \end{equation}
By Theorem \ref{BBWNTH} (see \eqref{MALIM2}),
\begin{equation} \label{EXNcAs} \lim_{N \to \infty}  \E_N
X_N^a(\vec c)=\int_M a(z)  {\bf 1}_{D \cap M(0)}  \det (dd^c \phi)(z) =\int_M a  {\rm MA}(\phi_{eq})). \end{equation}

We have
the following estimate of the variance:

\begin{lem} \label{variance} Let
$a \in C^{\infty}(M)$.   Then there exist constants $\alpha, \beta$ (see \eqref{G}) so that
\begin{equation} \label{E3} \E_N\left((X_N^a)^2\right) = \beta
\int_M \int_M (a(z)) (a(w))  d\nu(z)
d\nu(w) +  (\alpha - \beta) \int_{M}a(z) a(w)
\Delta \wedge d\mu_{\phi}
\end{equation} 
and so
$${\rm Var}(X_N^a) =
O(1).$$
\end{lem}

\begin{proof} We  let $f_j$ be a local representation of $S_{N,j}$
and let $\vec f$ be a local representation of $\vec S_N$. Then,

\begin{equation} \E_N\left((X_N^a)^2\right) = 
\int_M \int_M (a(z)) (a(w)) \left( \int_{S^{2d_N-1}}  |\langle \vec
f(z), \vec c \rangle|^2  |\langle \vec  f(w), \vec c \rangle|^2 d\mu_N(\vec c)\right) d\nu(z) d\nu(w).
\end{equation} 
Boundedness follows by the Schwartz inequality.
\end{proof}

We then write $\vec  f= |\vec f | u$ with $|u|\equiv
1.$ Then \begin{eqnarray*} |\langle \vec
f(z), \vec c \rangle|^2  |\langle \vec  f(w), \vec c \rangle|^2 &=& |\vec f(z)|^2
|\vec f(w)|^2 |\langle \frac{\vec
f(z)}{|\vec f(z)|}, \vec c \rangle|^2  |\langle  \frac{\vec  f(w)}{|\vec f(w) |}, \vec c \rangle|^2 \\ &&\\
& = & \Pi_{h^N, \nu}(z) \Pi_{h^N, \nu} (w) e^{- N(\phi(z) + \phi(w))}  |\langle \frac{\vec
f(z)}{|\vec f(z)|}, \vec c \rangle|^2  |\langle  \frac{\vec  f(w)}{|\vec f(w) |}, \vec c \rangle|^2.\end{eqnarray*} 
Let 
$\vec u_N(z) =  \frac{\vec
f(z)}{|\vec f(z)|}. $
Thus it suffices to calculate 
$$\E_N  |\langle \vec u_N(z), \vec c \rangle|^2  |\langle  \vec u_N(w), \vec c \rangle|^2 : = \E_N |Y_1|^2 |Y_2|^2, $$
where $Y_1=\langle c, \vec u_N(x)\rangle$,  $Y_2=\langle c,
\vec u_N(y)\rangle$.
To determine $\E(Y_1\overline Y_2)$, we note that for a random $s = s_N =
\sum c_j  S^N_j
\in H^0_N(M, L^N)$,
\begin{equation} \label{EPi} \E\left( s(z)\,\overline{s(w)}\right)
= \sum_{j,k=1}^{d_N} \E(c_j\bar c_k) \, S^N_j(z)\,\overline{
S^N_k(w)} = \sum_{j=1}^{d_N} S^N_j(z)\,\overline{
S^N_j(w)} =\Pi_{h^N, \nu}(z,w)\;.\end{equation}
For simplicity of notation we denote $\Pi_{h^N, \nu}$ by $\Pi_N$ in the remainder of the proof. 

Since $$\langle \vec c,\vec u_N(x)\rangle = \frac {\langle \vec c,\Psi_N(x)
\rangle} {\big|\wh \Psi_N(x)\big|}  = \frac {
s_N(z)}{\Pi_N(z,z)^{1/2}}\,,$$ we have by \eqref{EPi},
$$\E(Y_1\overline Y_2) = \frac
{\Pi_N(z,w)}{\Pi_N(z,z)^{1/2}\Pi_N(w,w)^{1/2}} : = P_N(z,w)\,.$$

\begin{lem} \label{varintb} Let $(Y_1,Y_2)$ be joint complex Gaussian
random variables with  mean 0 and  $\E(|Y_1|^2)= \E(|Y_2|^2)=1$.  Then
$$\E\big( |Y_1|^2\,  |Y_2|^2\big) = G\big(\left|\E(Y_1\overline
Y_2)\right|\big)\;,$$ where \begin{equation} \label{G} G(\cos \theta ):= \beta + (\alpha - \beta) \cos^2 \theta,
\end{equation} for certain universal $\alpha > \beta > 0$ \eqref{alphabeta}.\end{lem}

\begin{proof}  By replacing $Y_1$ with $e^{i\al}\,Y_1$, we can assume without
loss of generality that $\E(Y_1\overline Y_2)\ge 0$.   We can write
\begin{eqnarray*} Y_1 &=& \Xi_1\;,\\ Y_2 &=& (\cos\theta)\,\Xi_1 +
(\sin\theta)\,\Xi_2\,,
\end{eqnarray*} where $\Xi_1,\Xi_2$ are independent joint complex Gaussian
random variables with mean 0 and variance 1, and $\cos\theta=\E(Y_1\overline
Y_2)$.  Then (cf. \eqref{G}),
\begin{equation}\label{Elog} \E( |Y_1|^2  |Y_2|^2) =
G(\cos\theta)\;,\end{equation} where \begin{equation}\label{G2}
G(\cos\theta)=
\frac{1}{\pi^2} \int_{\C^2}  
|\Xi_{1}|^2\,\left|\Xi_1\cos\theta + \Xi_2\sin\theta
\right|^2\,e^{-(|\Xi_1|^2 + |\Xi_2|^2)}\, d\Xi_1
\,d\Xi_2\;.\end{equation}
We now verify that $G$ is given by \eqref{G}.
This is an elementary Gaussian calculation, but for the sake of completeness we go through the details.

Write
$\Xi_1 = r_1 e^{i \al},\ \Xi_2= r_2 e^{i(\al+\phi)}$, so that \eqref{G}
becomes
\begin{equation*}
G(\cos\theta)=\frac{2}{\pi}\int_0^{\infty} \int_0^{\infty}\int_0^{2\pi}
r_1 r_2
e^{-(r_1^2 + r_2^2)}  r_1^2  |r_1 \cos\theta+ r_2 e^{i
\phi}\sin\theta|^2\, d\phi\, dr_1\, dr_2\;.\end{equation*}  Evaluating the
inner integral, we obtain
\begin{equation*}  \int_0^{2\pi} |r_1 \cos \theta + r_2 \sin \theta e^{i
\phi}|^2  \,d\phi = 2\pi( r_1^2 \cos^2 \theta + r_2^2 \sin^2 \theta).  \end{equation*} Hence \begin{equation*} G(\cos\theta) =  2 \pi
\int_0^{\infty} \int_0^{\infty} r_1 r_2 e^{-(r_1^2 + r_2^2)} r_1^2  (r_1^2 \cos^2 \theta + r_2^2 \sin^2 \theta)\,dr_1 \,dr_2. \end{equation*}
We make the change of variables  $r_1 = \rho \cos \phi,\, r_2 = \rho
\sin\phi$
to get \begin{equation*} \begin{array}{lll} G(\cos\theta) & = &  
 4 \int_0^{\infty}
\int_0^{\pi/2} \rho^3 e^{-\rho^2} \rho^2 \cos^2 \phi)  (\rho^2
\cos^2 \phi \cos^2 \theta + \rho^2 \sin^2 \phi \sin^2 \theta)   \cos\phi\sin \phi
\,d\phi \,d\rho
\\ && \\ & = & 4 \cos^2 \theta \int_0^{\infty}
\int_0^{\pi/2} \rho^7 e^{-\rho^2}  \cos^5 \phi    \sin \phi
\,d\phi \,d\rho \\&&\\
& + &  4 \sin^2 \theta \int_0^{\infty}
\int_0^{\pi/2} \rho^7 e^{-\rho^2} \cos^3 \phi
 (1 - \cos^2 \phi) \sin \phi
\,d\phi \,d\rho \\ \\
& = & \alpha \cos^2 \theta +\beta \sin^2 \theta = \beta + (\alpha - \beta) \cos^2 \theta \;. \end{array} \end{equation*} 
where
\begin{equation} \label{alphabeta} \alpha = A \int_0^{\pi/2}   \cos^5 \phi    \sin \phi
\,d\phi = \frac{A}{6}, \;\; \beta = A \int_0^{\pi/2}  \cos^3 \phi
 (1 - \cos^2 \phi) \sin \phi d \phi = A (\frac{1}{4} - \frac{1}{6}), \end{equation}
 where $A = 4  \int_0^{\infty}
 \rho^7 e^{-\rho^2} d\rho = 2 \int_0^{\infty} x^3 e^{-x} dx = 2 (3!). $
\end{proof}

It follows that 

\begin{equation} \label{E2} \E_N\left((X_N^a)^2\right) = 
\int_M \int_M (a(z)) (a(w)) (\beta + (\alpha - \beta) P_N^2(z,w) ) d\nu(z)
d\nu(w). 
\end{equation} 

We now useTheorem  \ref{BERMANOFFD} to get
$$\int_M P_N^2(z,w) a(z) a(w) d\nu(z) d\nu(w)
\to \int_{M}a(z) a(w)
\Delta \wedge d\mu_{\phi}. $$
Combining this limit formula with \eqref{E2} proves \eqref{E3}, and concludes
the proof of Lemma \ref{varintb}.

\subsection{$4$th moment bounds}

To prove almost sure convergence to zero of the random variables
$$Y_N(\vec c): = \left| X_N^a(\vec c) - \int_M a \; {\rm MA}(\phi_{eq}) \right|^2$$
we need to show that the variances of these random variables are bounded.
The variance of $Y_N$ is a fourth moment and should not be confused
with \eqref{VARXN}.

\begin{lem} $\rm{Var}(Y_N) \leq C. $ \end{lem}

\begin{proof} Since 
$\rm{Var}(Y_N) = \E (|Y_N - \E Y_N|^2) = \E Y_N^2 - (E Y_N)^2$
and $\E Y_N \to 0$ it suffices to show that 
$E Y_N^2 =  \E_N \left| X_N^a(\vec c) - \int_M a \; {\rm MA}(\phi_{eq})) \right|^4$ is uniformly bounded. In fact it suffices to show that $\E |X_N^a(\vec c)|^4$
is uniformly bounded, which is a simple calcluation of Gaussian integrals
and is obvious in the spherical model of Section \ref{GFS}.

\end{proof}

It then follows by the Kolmogorov strong law of large numbers that
$$\frac{1}{K} \sum_{N \leq K} Y_N  \to 0, \;\; \rm{almost\;surely}.$$
Since $Y_N > 0$ this is enough to give a subsequence of indices
$N_k$  of density
one for which $Y_{N_k} \to 0$. This concludes the proof of Theorem \ref{RANSEQ}.

   Further remarks on this step are given
in the next section.

 \section{Quantum ergodicity of random orthonormal bases: Proof of Theorem \ref{RANONB} }

In this section we prove that sequences of random ONB's of $H^0(M, L^N)$
are quantum ergodic. The proof follows the same lines as in \cite{ShZ99}, so we mainly emphasize what changes in the proof if we use the general
data $(h, \nu)$ to define ONB's and quantum ergodicity. The main change is in the use of the \szego limit formula, which now has to be applied to Toeplitz operators relative to non-standard Bergman projections. We only need the second moment calculation, which as in the previous section follows from the off-diagonal Bergman kernel asymptotics of Theorem \ref{BERMANOFFD}.

\subsection{\szego limit formulae for Toeplitz operators}

We abreviate $\Pi_{h^N, \nu}$ by $\Pi_N$  define the  Toeplitz operator
with multipler $g \in C(M)$ by  $$T^g_N = \Pi_N
M_g \Pi_N = \Pi_N M_g :H^0(M, L^N) \to H^0(M, L^N),$$ where $M_g$ is
multiplication by
 $g$.
 Then
$T^g_N$ is a self-adjoint operator 
$H^0(M, L^N)$, which can be
identified with a Hermitian $d_N\times d_N$ matrix by fixing  one orthonormal basis.  Here, as above, $d_N = \dim H^0(M, L^N)$.

\begin{lem} \label{taulem}
$$\lim_{N \to \infty} \frac{1}{d_N}{\rm
Tr}\;T^g_N   = \int_M g(z) d\mu_{\phi} (z). $$
\end{lem}

\begin{proof} The trace is obviously given by
$$\tau_{h, \nu}(g) = \lim_{N \to \infty} \frac{1}{N^m} \int_M g(z)
\Pi_N (z) d\nu(z). $$ Hence the result follows from
Berman's asymptotics Theorem \ref{Ber1} and \eqref{MALIM2}. \end{proof}

Henceforth we use the notation,

\begin{defin} $\tau_{h, \nu}(g) : =
 \int_M g(z) d\mu_{\phi} (z).$ \end{defin}

\begin{lem} \label{SQUARE}  Under the assumptions of Theorem \ref{BERMANOFFD},
$$ \lim_{N \rightarrow \infty}
\frac{1}{d_N} {\rm Tr}\;(T^g_N)^2 = \tau_{h, \nu} (g^2)\;.$$
\label{BGlemma}\end{lem}

\begin{proof} By Theorem \ref{BERMANOFFD}, we have

$$\begin{array}{lll}N^{-m}  {\rm Tr}\;(T^g_N)^2 & = & N^{-m} {\rm Tr} M_g \Pi_N M_g \Pi_N = N^{-m} \int_M \int_M g(z)    g(w) |\Pi_N(w, z)|^2 d\nu(w)
d\nu(z) \\ && \\  & \to &  \int_M \int_M g(z) g(w) \Delta \wedge
d\mu_{\phi}
= \int_M g(z)^2 d\mu_{\phi}(z).
\end{array}$$

\end{proof}

\subsection{Proof of quantum ergodicity}

\noindent {\it Definition:\/} We say that ${\bf S} \in {\mathcal
ONB}$ has the ergodic property  if 
$$\lim_{N \rightarrow \infty}\frac{1}{N}
\sum_{n=1}^N \frac{1}{d_n} \sum_{j = 1}^{d_n} \left|\int_M g(z)
\|S^n_j(z)\|^2_{h_n}d \nu - \tau_{h, \nu} (g)
\right|^2=0\;,\quad \forall g \in {\mathcal C}(M)\;.
\eqno({\mathcal EP})$$
It may seem unaesthetic to average in $n$ as well as over $j$, but
averages of positive quantities which tend to zero can only happen when
`almost all' of the corresponding terms tend to zero. The double-average slightly
weakens the notion of `almost all'. In complex dimensions $\geq 2$, it
is unnecessary to double average.

More precisely, the ergodic  property may be rephrased in the following way: Let ${\bf S}
= \{(S^N_1, \dots, S^N_{d_N}): N=1,2,\ldots\}\in {\mathcal ONB}$.
Then the ergodic property $({\mathcal EP})$ is equivalent to the
following weak* convergence property: {\it There exists a
subsequence $\{S'_1,S'_2,\ldots\}$ of relative density one of the
sequence $\{S^1_1, \dots, S^1_{d_1},\ \dots,\
S^N_1,\dots,S^N_{d_N},\ \dots\}$ such that}
$$\int_M g(z) \|S'_n (z) \|^2
d\nu   \rightarrow \tau_{h, \nu}(g) \;,\quad \forall g \in
{\mathcal C}(M)\;. \eqno({\mathcal EP}')$$ A subsequence
$\{a_{k_n}\}$ of a sequence $\{a_n\}$ is said to have relative
density one if $\lim_{n\to\infty} n/k_n = 1$.  We refer to
\cite{ShZ99} for the (standard)  proof of the equivalence of the notions. We now
generalize the result of \cite{ShZ99} to our setting. Recall the definition 
of $ {\mathcal ONB}$ in
\eqref{ONB}.

\begin{theo}\label{EONB} Let $(L, h) \to M$ be a  line bundle with $c_1(L)$ a \kahler
class, let $h$ be a smooth Hermitian metric on $L$ and let $\nu
$ be a Bernstein-Markov probability measure  (\ref{BM}).
 Then

  {\bf (a)} A random ${\bf S}\in{\mathcal ONB}$ has the
ergodic property $({\mathcal EP})$, or equivalently, $({\mathcal
EP}')$. In fact, in complex dimensions $m \geq 2$, a random ${\bf
S}\in{\mathcal ONB}$ has the property $$\lim_{N \rightarrow
\infty} \frac{1}{d_N} \sum_{j=1}^{d_N} \left| \int_M g\|S^N_j\|^2
dV - \tau_{h, \nu}(g)\right|^2 = 0\;,\quad \forall g\in {\mathcal
C}(M)\;,$$ or equivalently, for each $N$ there exists a subset
$\Lambda_N \subset \{1, \dots, d_N\}$ such that $ \frac{\#
\Lambda_N}{d_N} \rightarrow 1$ and $$\lim_{N \rightarrow \infty, j
\in \Lambda_N} \int_M g \|S^N_j\|^2 d\nu = \tau_{h, \nu}(g).$$\\ {\bf (b)}
A random sequence of sections ${\bf
s}=\{s_1,s_2,\dots\}\in{\mathcal S}$ has a subsequence
$\{s_{N_k}\}$ of relative density 1 such that $$\int_M g(z)
\|s_{N_k} (z) \|^2 d\nu \rightarrow \tau_{h, \nu}(g)\;,\quad \forall
g \in {\mathcal C}(M)\;.$$ In complex dimensions $m \geq 2$, the
entire sequence has this property.
\end{theo}

 To simplify the notation, we write
\begin{equation}\label{A} A^g_{nj}({\bf S}) =\left|\int_M g(z)
\|S^n_j(z)\|^2_{h^n}d\nu- \tau_{h,
\nu}(g)\right|^2\;.\end{equation}

\begin{proof} We adapt the
 proof in the case of positive line bundles to our more general setting. For the reader's convenience
 we also recall the main steps of the proof even when they are unchanged in the present setting.

  We
then have
\begin{equation} A^g_{nj}({\bf S}) = \left|(g S^n_j,
S^n_j)- \tau_{h, \nu}(g) \right|^2 = \left|(T^g_N S^n_j,
S^n_j)-\tau_{h, \nu}(g) \right|^2 = \left|(U^*_n T^g_n U_n e^n_j,
e^n_j)- \tau_{h, \nu}(g) \right|^2,\end{equation} where ${\bf S} =
\{U_N\},\ U_N\in \U(d_N)\equiv{\mathcal ONB}_N$.

By Lemma \ref{taulem},
\begin{equation}\label{Anj'} A^g_{nj}({\bf S})=\wt A^g_{nj}({\bf
S})+O(\frac{1}{n})\;,\end{equation} where
\begin{equation}\label{Anj}\wt A^g_{nj}({\bf S})=\left|(U^*_n
T^g_n U_n e^n_j, e^n_j)-\frac{1}{d_n}{\rm Tr}\;T^g_n \right|^2
\;.\end{equation}  (The bound for the $O(\frac{1}{n})$ term in
(\ref{Anj'}) is independent of $\bf S$.)

Once we fix an orthonormal basis, the  skew-Hermitian operator $iT^g_n$ can be identified with an element of the Lie
algebra ${\go u}(d_N)$ of the unitary group $\U(d_N)$. Let ${\go t}(d)$ denote the
Cartan subalgebra of diagonal elements in ${\go u}(d)$, and let
$\|\cdot\|^2$ denote the Euclidean inner product on ${\go t}(d)$.
Also let
$$J_d:i{\go u}(d)\rightarrow i{\go t}(d)$$
denote the orthogonal projection (extracting the diagonal).
Finally, let $$\bar J_d (H)=\left(\frac{1}{d}{\rm
Tr}\;H\right){\rm Id}_{d}\;,$$ for Hermitian matrices $H\in i{\go
u}(d)$. (Thus, $H=H^0 +\bar J_d(H)$, with $H^0$ traceless, gives
us the decomposition  ${\go u} (d) =  {\go su} (d)\oplus \R$.) 

As discussed below the statement of Theorem \ref{RANSEQ}, we identify
a random sequence of ONB's of $H^0(M, L^N)$ with a random sequence
$\{U_n\}$ of unitary matrices with respect to the fixed ONB $\{e^n_j\}_{j=1}^{d_n}$. The infinite product $ {\mathcal ONB}: =\prod_{n} U(d_n)$ is endowed with normalized product Haar measure.
We introduce the random variables:
$$\begin{array}{l}
Y^g_n: {\mathcal ONB} \to [0,+\infty)\\[6pt]
Y^g_n({\bf S}) := \|J_{d_n} (U_n^*T^g_n U_n) - \bar
J_{d_n}(T^g_n)\|^2\end{array}$$ By (\ref{Anj'})
\begin{equation}\label{Y}\frac{1}{d_n} Y^g_n({\bf S})
=\frac{1}{d_n}\sum_{j=1}^{d_n}\wt A^g_{nj}({\bf S})
=\frac{1}{d_n}\sum_{j=1}^{d_n}A^g_{nj}({\bf S})
+O(\frac{1}{n})\end{equation} (where the $O(\frac{1}{n})$ term is
independent of {\bf S}). Thus,  $({\mathcal EP})$ is equivalent
to:
\begin{equation}\label{EP*}
\lim_{N\rightarrow\infty}\frac{1}{N}\sum_{n=1}^N \frac{1}{d_n}
Y^g_n ({\bf S}) = 0\;,\quad \forall g \in {\mathcal C}(M)\;.
\end{equation}  We plan to apply the Kolmogorov strong law of
large numbers to these sums of  independent random variables, and to prove (\ref{EP*}) we need the following asymptotic formula for the expected values of
the $Y^g_n$.

\begin{lem} \quad $\displaystyle E(Y^g_n) = \tau_{h, \nu} (g^2) -
(\tau_{h, \nu} (g))^2 +o(1)\;.\label{BGcor}$\end{lem}

 Before proving Lemma \ref{BGcor}, we show how it implies Theorem
 \ref{EONB}. First, the Lemma implies that
\begin{equation}\label{EP*ave}\lim_{N\rightarrow \infty} \frac{1}{N}
\sum_{n=1}^N E\left(\frac{1}{d_n}Y^g_n\right)=0\;,\end{equation}
since  $\frac{1}{N}
\sum_{n=1}^N \frac{1}{d_n} \to 0$. 
Next we observe that the individual terms have bounded variances:
$$ \mbox{\rm Var}\left(\frac{1}{d_n}Y^g_n\right)
\leq  \sup \left(\frac{1}{d_n}Y^g_n\right)^2 \leq \max_j\; \sup
(\wt A^g_{nj})^2\;.$$ By (\ref{Anj}),$$ \wt A^g_{nj}({\bf S}) \leq
4(U^*_n T^g_n U_n e^n_j,e^n_j)^2\leq 4\sup g^2 \;,$$ and therefore
\begin{equation} \mbox{\rm
Var}\left(\frac{1}{d_n}Y^g_n\right)\leq 16 \sup g^4
<+\infty\;.\end{equation} Since the variances of the independent
random variables $\frac{1}{d_n}Y^g_n$ are bounded, (\ref{EP*})
follows from (\ref{EP*ave}) and the Kolmogorov strong law of large
numbers, which gives part (a) for general dimensions.

  \begin{rem}  In
dimensions $m \geq 2$, we obtain the improved conclusion as
follows: From the fact that $E(\frac{1}{d_N} Y^{g}_N) =
O(\frac{1}{N^m})$ it follows that
$E\left(\sum_{N=1}^{\infty}\frac{1}{d_N}Y^{g}_N\right)<+\infty$
and thus $\frac{1}{d_N} Y^{g}_N \rightarrow 0$ almost surely when
$m \geq 2$. The quantity we are interested in is
$$X_N^{g}: = \frac{1}{d_N} \sum_{j=1}^{d_N} \left| \int_M g
\|S^N_j\|^2 dV - \tau_{h,\nu}(g)  \right|^2 = \frac{1}{d_N}
\sum_{j=1}^{d_N} A^{g}_{Nj}.$$ However, by (\ref{Y}),
$$\sup_{{\mathcal ONB}} |X^{g}_N - \frac{1}{d_N} Y^{g}_N| =
O(\frac{1}{N}).$$ Hence also $X_N^{g} \rightarrow 0$ almost
surely.
\end{rem}

To verify part (b), we note that since $E(\wt A^g_{nj})=E(\wt
A^g_{n1})$, for all $j$, it follows from (\ref{Y}) that $E(\wt
A^g_{n1})= E(\frac{1}{d_n}Y^g_n)$.  Thus,  \begin{equation}
\lim_{N\to\infty}\frac{1}{N}\sum_{n=1}^N \wt
A^g_{n1}=0\;,\label{EP1}\end{equation} or equivalently,
\begin{equation} \lim_{N\to\infty}\frac{1}{N}\sum_{n=1}^N
A^g_{n1}=0\;.\label{EP*1}\end{equation} Part (b) then follows from
(\ref{EP*1}) exactly as before.

It remains to prove Lemma~\ref{BGcor}. Denote the eigenvalues of
$T^g_n$ by $\lambda_1,\dots,\lambda_{d_n}$ and write
$${\mathcal S}_k(\lambda_1,\dots,\lambda_{d_n})=\sum_{j=1}^{d_n}
\lambda_j^k\;.$$ Note that  \begin{equation} {\rm Tr}\;(T^g_n)^k
={\mathcal S}_k(\lambda_1,\dots,\lambda_{d_n})\;.
\end{equation}
Lemma~\ref{BGcor} is an immediate consequence of
Lemma~\ref{BGlemma} and the following formula:

\begin{equation}\label{orbit-int} \int_{\U(d)} \|J_d (U^*D(\vec\lambda) U) -
\bar J_d(D(\vec\lambda))\|^2 dU = \frac{{\mathcal
S}_2(\vec\lambda)}{d+1}- \frac{{\mathcal
S}_1(\vec\lambda)^2}{d(d+1)}\;,\end{equation} where
$\vec\lambda=(\lambda_1,\dots,\lambda_d)\in\R^d$, $D(\vec\lambda)$
denotes the diagonal matrix with entries equal to the $\lambda_j$,
and integration is with respect to Haar probability measure on
$\U(d)$.

  We refer to \cite{ShZ99} for the proof of  (\ref{orbit-int}).

\end{proof}

\end{document}